\documentclass[11pt]{amsart}
\usepackage{epsf,overpic,amssymb,amscd,times,stmaryrd,euscript,times,psfrag}
\usepackage{tikz-cd}
\usepackage{hyperref}

\newtheorem{thm}{Theorem}[section]
\newtheorem*{thm*}{Theorem}

\newtheorem{claim}[thm]{Claim}
\newtheorem{example}[thm]{Example}
\newtheorem{pb}[thm]{Problem}

\newtheorem{prop}[thm]{Proposition}

\newtheorem{defn}[thm]{Definition}

\newtheorem{lemma}[thm]{Lemma}
\newtheorem{cor}[thm]{Corollary}
\newtheorem{conj}[thm]{Conjecture}

\newtheorem{rmk}[thm]{Remark}
\newtheorem{fact}[thm]{Fact}

\numberwithin{equation}{subsection}

\newcommand{\R}{\mathbb{R}}
\newcommand{\Z}{\mathbb{Z}}
\newcommand{\Q}{\mathbb{Q}}

\newcommand{\B}{\mathcal{B}}
\newcommand{\F}{\mathcal{F}}

\newcommand{\bdry}{\partial}

\newcommand{\hh}{\mathfrak{h}}

\newcommand{\op}{\operatorname}

\newcommand{\bs}{\boldsymbol}

\makeatletter
\newcommand*\bigcdot{\mathpalette\bigcdot@{.5}}
\newcommand*\bigcdot@[2]{\mathbin{\vcenter{\hbox{\scalebox{#2}{$\m@th#1\bullet$}}}}}
\makeatother

\newcommand{\be}{\begin{enumerate}}
\newcommand{\ee}{\end{enumerate}}

\begin{document}

\title[Foliations, contact structures and their interactions]{Foliations, contact structures and their interactions in dimension three}

\author{Vincent Colin}
\address{Universit\'e de Nantes, UMR 6629 du CNRS, 44322 Nantes, France}
\email{vincent.colin@univ-nantes.fr}

\author{Ko Honda}
\address{University of California, Los Angeles, Los Angeles, CA 90095}
\email{honda@math.ucla.edu} \urladdr{http://www.math.ucla.edu/\char126 honda}

\date{}



\thanks{VC supported by ANR Quantact and KH supported by NSF Grants DMS-1406564 and DMS-1549147.}

\begin{abstract}
We survey the interactions between foliations and contact structures in dimension three, with an emphasis on sutured manifolds and invariants of sutured contact manifolds. This paper contains two original results: the fact that a closed orientable irreducible $3$-manifold $M$ with nonzero second homology carries a hypertight contact structure and the fact that an orientable, taut, balanced sutured $3$-manifold is not a product if and only if it carries a contact structure with nontrivial cylindrical contact homology. The proof of the second statement uses the Handel-Miller theory of end-periodic diffeomorphisms of end-periodic surfaces.
\end{abstract}

\maketitle

\tableofcontents

\section{Introduction}

Codimension one foliations are a powerful tool for probing the topology of an ambient manifold of dimension three. Since the work of Gabai \cite{Ga}, {\em taut} foliations have been identified as a particularly relevant class of foliations, at the epicenter of many breakthroughs such as the resolution of the Property $R$ and $P$ conjectures for knots respectively by Gabai \cite{Ga1} and Kronheimer-Mrowka~\cite{KM}. The theory of sutured manifolds and their decompositions \cite{Ga} provides a (finite depth) taut foliation on every closed orientable irreducible $3$-manifold with nontrivial integral second homology. The existence of a taut foliation on an irreducible $3$-manifold $M$ is currently conjectured to be equivalent to the left-orderability of the fundamental group of $M$ \cite{BGW}.

Contact structures are more flexible in nature, but the results of Eliashberg-Thurston~\cite{ET} and Vogel~\cite{Vo1} give a correspondence between the two worlds, namely an essentially unique way to deform a taut foliation without a torus leaf into a {\em tight} contact structure. The stability of contact structures makes them appear as a discrete version of foliations. It is also important that tight contact structures might exist when taut foliations do not (for example on $S^3$), extending the possible range of investigations.

More interesting is that they are well-adapted to pseudo-holomorphic curve techniques. Amongst those, two invariants emerge: embedded contact homology and  (cylindrical) contact homology. Both are homologies of chain complexes generated by finite products of periodic orbits of a Reeb vector field for a contact structure and whose differentials count certain pseudo-holomorphic curves in the symplectization of the contact manifold. Embedded contact homology (ECH) was defined by Hutchings \cite{Hu1,Hu2,Hu3} and Hutchings-Taubes \cite{HT1,HT2} and was shown to be a topological invariant \cite{T2} isomorphic to the Heegaard Floer homology of Ozsv\'ath-Szab\'o \cite{OSz1,OSz2} in \cite{CGH0}-\cite{CGH3}. Contact homology, proposed by Eliashberg-Givental-Hofer~\cite{EGH} and defined in full generality by Pardon~\cite{Par} and Bao-Honda~\cite{BH}, is a contact invariant which is sensitive to the particular choice of contact structure on a given manifold.

The theory of convex surfaces in contact $3$-manifolds, developed by Giroux~\cite{Gi1}, is also surprisingly close to the theory of sutured $3$-manifolds, invented by Gabai to construct taut foliations. Indeed both interact in a nested way, making it possible to define invariants of sutured contact manifolds and develop a gluing theory for contact manifolds parallel to that of foliated sutured ones \cite{HKM1}-\cite{HKM6}, \cite{Co1,Co2}, \cite{CH, CGHH}. This combination of contact geometry and foliation theory through its holomorphic invariants --- and in particular the {\em contact class} associated to a contact structure in the Heegaard Floer group~\cite{OSz3} --- has expanded the applicability of Gabai's work and provided striking new results such as the proof of the Gordon conjecture by Kronheimer, Mrowka, Ozsv\'ath and Szab\'o~\cite{KMOS} or the characterization of fibered knots by Ghiggini~\cite{Gh} and Ni~\cite{Ni}.  The existence of a taut foliation on an irreducible rational homology sphere $M$ is now conjectured to have a complete characterization in terms of Heegaard Floer homology (or equivalently ECH): the rank of the Heegaard Floer hat group of $M$ is strictly larger than the order of $H_1 (M;\Z)$.

The goal of this paper is to survey these interactions, with an emphasis on sutured manifolds and invariants of sutured contact manifolds.  We also discuss some open problems. The paper contains two original results: the fact that a closed orientable irreducible $3$-manifold $M$ with nonzero second homology carries a hypertight contact structure (Theorem~\ref{thm: hypertight}) and the fact that an orientable, taut, balanced sutured $3$-manifold is not a product if and only if it carries a contact structure with nontrivial cylindrical contact homology (Theorem~\ref{thm: cylindrical}). The proof of Theorem~\ref{thm: cylindrical} uses the Handel-Miller theory of end-periodic diffeomorphisms of end-periodic surfaces. We also discuss work in progress with Ghiggini and Spano \cite{CGHS} to prove the isomorphism between the sutured versions of Heegaard Floer homology and ECH.

\section{A brief overview of foliations and contact structures}

\subsection{Definitions}

On a $3$-manifold $M$, there are two classes of homogeneous plane fields. Given a local nonvanishing $1$-form $\alpha$ with kernel a homogeneous plane field $\xi$, either the $3$-form $\alpha \wedge d\alpha$ vanishes everywhere or is nonzero everywhere. In the first case, the plane field $\xi$ integrates into a {\em foliation} and is locally defined by an equation $dz=0$ in $\R^3$. It is the tangent space to the local surfaces $\{ z=\mbox{const}\}$. In the second case, it is locally given by an equation $dz-ydx=0$ and is a {\em contact structure}. In these coordinates, the vector field $\partial_y$ is tangent to the contact plane and the slope of the line field $\xi \cap \{dy=0\}$  in the $(x,z)$ coordinates is $y$: the plane field $\xi$ rotates with $y$, a manifestation of its maximal nonintegrability. When $\xi$  is a contact structure, the sign of the $3$-form $\alpha \wedge d\alpha$ is independent of the choice of $\alpha$ and hence $\xi$ induces an orientation of $M$.

In this paper, we assume that all plane fields are cooriented and $3$-manifolds oriented.

The topology of a $3$-manifold is often tested by looking at its submanifolds of dimensions $1$ and $2$. In the presence of a plane field $\xi$, one distinguishes horizontal curves (i.e., curves everywhere tangent to $\xi$) and transverse ones. Whether $\xi$ is integrable or contact makes a huge difference: every horizontal curve of a foliation must stay in a leaf, whereas any two points can be connected by a horizontal (or transverse) arc in a contact manifold. If $\xi$ is contact, a horizontal curve will usually be called {\em Legendrian}. Also in the contact case, there is a preferred class of transverse vector fields called {\em Reeb vector fields} which are the ones whose flow preserves $\xi$.

If $S$ is a surface in $(M,\xi)$, we define its {\em characteristic foliation} to be the singular foliation $\xi S$ of $S$ generated by the singular line field $\xi\cap TS$. The singularities are the points $p\in S$ where $\xi_p=T_pS$. For a generic $S$, they are isolated. Generically the singularities of $\xi S$ are either saddles or centers when $\xi$ is integrable, and saddles or sources/sinks when $\xi$ is a contact structure. If both $\xi$ and $S$ are oriented, then the foliation $\xi S$ is also oriented.

The main tool for analyzing a $3$-dimensional contact manifold is convex surface theory, which was developed by Giroux~\cite{Gi1}. When $\xi$ is a contact structure, a closed surface $S\subset (M,\xi)$ is {\em convex} if there exists a vector field $X$ transverse to $S$ and whose flow preserves $\xi$. The convexity condition is a $C^\infty$-generic condition. If $S$ is convex, the {\em dividing set}
$$\Gamma_S =\{ x\in S~|~ X(x)\in \xi(x)\}$$
is an oriented embedded $1$-submanifold of $S$ (i.e., a multicurve) transverse to the characteristic foliation $\xi S$. It decomposes $S$ into regions $R_+$ and $R_-$ such that $X$ is positively (resp.\ negatively) transverse to $\xi$ on $R_+$ (resp.\ $R_-$). It turns out that the dividing set $\Gamma_S$ does not depend on the choice of the vector field $X$ up to isotopy through curves transverse to the characteristic foliation $\xi S$. The dividing set $\Gamma_S$ completely determines the germ of $\xi$ near $S$. When $\partial M$ is convex, the dividing set $\Gamma_{\partial M}$ gives $M$ the structure of a sutured manifold; see Section~\ref{section: suture}.

\subsection{Flexibility vs.\ rigidity}

We review the boundary between flexibility and rigidity both for foliations and contact structures.

\subsubsection{Foliations}

We first discuss the situation for foliations.  A {\em Reeb component} is a foliation of the solid torus $S^1 \times D^2$ such that the boundary $S^1\times \bdry D^2$ is a leaf and $int(S^1\times D^2)$ is foliated by an $S^1$-family of planes which ``winds around" in such a way that the characteristic foliation of each meridian disk $\{\theta\} \times D^2$ consists of concentric circles with one singularity at the center. The presence of Reeb components makes foliations flexible and subject to an $h$-principle. In fact Thurston~\cite{Th} showed that on a given $3$-manifold every plane field is homotopic to an integrable one, typically with many Reeb components. Along the same lines, by inserting Reeb components Eynard-Bontemps~\cite{E} showed that two integrable plane fields that are homotopic as plane fields are homotopic through integrable plane fields.

One gets more rigidity by considering foliations without Reeb components, or the more restrictive class of {\em taut} foliations, for which there is a closed transverse curve that passes through every leaf.  This condition prevents Reeb components from existing. The existence of a taut foliation imposes strong restrictions on the ambient manifold: its universal cover is $\R^3$ by Palmeira \cite{Pa} and all the leaves are $\pi_1$-injective by Novikov \cite{No}.  An equivalent definition for a foliation to be taut is to have a volume-preserving transverse vector field. Moreover, a vector field transverse to a taut foliation has no contractible periodic orbit. Finally, taut foliations only exist in  a finite number of homotopy classes of plane fields \cite{Ga2}.

\subsubsection{Contact structures}

On the contact side, a contact structure $\xi$ is {\em overtwisted} if there exists an embedded disk $D$ such that $\xi=TD$ along $\bdry D$; otherwise it is {\em tight}. In this paper we emphasize two subclasses of tight contact structures, mostly due to their close connections with taut foliations:
\be
\item {\em universally tight} contact structures, i.e., tight contact structures that remain tight when pulled back to the universal cover; and
\item {\em hypertight} contact structures, i.e., contact structures that admit Reeb vector fields with no contractible orbit.
\ee
By a theorem of Hofer~\cite{Ho} that states that an overtwisted contact form on a closed manifold always admits a contractible Reeb orbit, a hypertight contact structure is always universally tight. The hypertightness condition also insures that the cylindrical contact homology is well-defined and hypertight contact forms are typically easier to deal with when computing holomorphic invariants; see Section~\ref{section: invariant}.

Similar to foliations with Reeb components, overtwisted contact structures are subject to an $h$-principle which was discovered by Eliashberg~\cite{El}: there is exactly one homotopy class of overtwisted contact structures in every homotopy class of $2$-plane fields.  However, rigidity shows up again when considering higher homotopy groups in the space of overtwisted contact structures. For example, we have the following:

\begin{thm}[Chekanov-Vogel~\cite{Vo2}]
If $\xi$ is an overtwisted contact structure of $S^3$ whose Hopf number is $1$, then there exists a loop of overtwisted contact structures based at $\xi$ that is contractible in the space of plane fields, but not in the space of overtwisted contact structures.
\end{thm}

\begin{pb}
Does the same phenomenon hold for foliations with Reeb components?
\end{pb}

Tight contact structures have strong rigidity properties and admit a roughly classification in dimension three~\cite{Co4,CGiH,HKM2}: They exist in a finite number of homotopy classes of plane fields. If $M$ is atoroidal (e.g., hyperbolic), it carries finitely many --- this number may be zero --- tight contact structures up to isotopy. If $M$ is toroidal and irreducible, it carries infinitely many tight contact structures up to diffeomorphism. Finally, the class of tight contact structures is stable under connected sum \cite{Co1}.

\begin{pb}
Does every hyperbolic $3$-manifold carry a tight contact structure?
\end{pb}

Compare this to the situation for taut foliations where Roberts-Shareshian-Stein \cite{RSS} showed that there are infinitely many hyperbolic $3$-manifolds without taut foliations; there are other obstructions due to Calegari-Dunfield~\cite{CD} and Kronheimer-Mrowka-Ozsv\'ath-Szab\'o~\cite{KMOS}.

\section{From foliations to contact structures}

\subsection{Perturbing foliations and torsion}

The link between foliations and contact structures was discovered by Eliashberg and Thurston \cite{ET}.

\begin{thm}[Eliashberg-Thurston]\label{thm: eliashberg-thurston}
Every $C^2$ foliation on a closed $3$-manifold different from the foliation by spheres of $S^1 \times S^2$ is a limit of a sequence of positive contact structures and also a sequence of negative contact structures. Moreover, if the foliation is taut, then every contact structure $C^0$-close to it is semi-fillable and universally tight.
\end{thm}

A contact $3$-manifold $(M,\xi)$ is {\em semi-fillable} if there exists a symplectic manifold $(W,\omega)$ with contact boundary $\bdry W= (M,\xi)\sqcup (M',\xi')$ such that $\omega|_{\xi\cup \xi'}$ is symplectic. The symplectic manifold $(W,\omega)$ is called a {\em semi-filling of $(M,\xi)$.}

Theorem~\ref{thm: eliashberg-thurston} was extended by Bowden \cite{Bo1} to the case of $C^0$ foliations and it was shown in \cite{Bo2,Co5} that every contact structure close to a Reebless foliation is universally tight. In the presence of a torus leaf, Giroux noticed that the {\em torsion} phenomenon could provide different contact structures approximating the same foliation. The following is the basic example:

\begin{example}
On the $3$-torus
$$T^3 =\{ (x,y,t) \in \R/\Z \times \R/\Z\times \R/(2\pi \Z)\},$$
for any positive integer $n$ and real number $\epsilon\neq 0$, the plane field
$$\xi_n^{\epsilon} =\ker(dz+\epsilon(\cos (nx)dt-\sin(ny)dt))$$
is a contact structure, positive when $\epsilon >0$ and negative when $\epsilon<0$. By Gray's stability theorem, for a fixed $n$ all the structures $\xi_n^{\epsilon}$ with $\epsilon >0$ are isotopic, and similarly for $\epsilon<0$. Moreover, by Giroux \cite{Gi2}, two different $n$ give two nondiffeomorphic contact structures. Finally we observe that when $\epsilon$ goes to $0$, all these contact structures converge to the integrable plane field $\xi_n^0=\ker dz$.
\end{example}

More generally, let $\xi_n$ be the contact structure defined on the thickened torus
$$T^2\times I=\{ (x,y,t) \in \R/\Z \times \R/\Z \times [0,2\pi]\}$$
by the equation $\cos tdx-\sin tdy=0$.
We define the {\em torsion $\tau_{(M,\xi)}$} of a contact manifold $(M,\xi)$ to be the supremum of the integers $n$ for which there exists a contact embedding
$$(T^2\times I,\xi_n) \hookrightarrow (M,\xi).$$
One can also specify the torsion $\tau_\eta$ in a given isotopy class $\eta$ of embeddings of $T^2$. It is immediate from Eliashberg's classification of overtwisted contact structures that $\tau_\eta$ is infinite for every $\eta$. When $\xi$ is tight, $\tau_\eta=0$ in every compressible class $\eta$.

The following is still open:

\begin{pb}
Prove that $\tau_{(M,\xi)}<\infty$ whenever $(M,\xi)$ is tight.
\end{pb}

It was shown in \cite{Co3} that $\sup_{\eta\in \mathcal{D}} \tau_\eta <\infty$ when $(M,\xi)$ is universally tight,  where
$\mathcal{D}$ is the set of isotopy classes of incompressible tori that do not appear in the Jaco-Shalen-Johannson (JSJ) decomposition of $M$. However, we still do not know whether $\tau_\eta<\infty$ for the JSJ classes and whether $\tau_{(M,\xi)}<\infty$ if $(M,\xi)$ is tight but not universally tight.

It was conjectured in \cite{Co6} that the torsion phenomenon which appears in the presence of a torus leaf was the only reason for the nonuniqueness of the approximation. The answer to this problem was given by Vogel:

\begin{thm}[Vogel \cite{Vo1}] \label{thm: uniqueness}
If $\F$ is a foliation on a closed $3$-manifold which does not have a torus leaf, is not a foliation by planes, and is not a foliation by cylinders, then it has a $C^0$-neighborhood in the space of $2$-plane fields in which all positive (or negative) contact structures are isotopic.
\end{thm}

The strength of this result is that it implies that the invariants of a contact approximation such as contact homology immediately become invariants of the foliation.

It was already known by Honda-Kazez-Mati\'c \cite{HKM3} that:

\begin{thm} \label{thm: unique tight for fibered}
If $\pi:M\to S^1$ is a fibered hyperbolic $3$-manifold (and hence has pseudo-Anosov monodromy),  then there is a unique positive contact structure $\xi_\pi$ up to isotopy that is homotopic to $T\mathcal{F}_\pi$, where $\mathcal{F}_\pi$ is the foliation by fibers.
\end{thm}

The contact homology of $\xi_\pi$ is thus an invariant of the pseudo-Anosov map.

\begin{pb}\label{pb: fibration}
Compute the contact homology in this case. What does it tell us about the pseudo-Anosov map?   Study its behavior under composition of monodromies, provided they stay pseudo-Anosov.  A first step was taken in Theorem~\ref{thm: fibration}, which provides a convenient Reeb vector field.
\end{pb}

\begin{pb}\label{pb: fibration2}
Given a fibration $\pi:M\to S^1$ with pseudo-Anosov monodromy, is every taut foliation which is homotopic in the space of $2$-plane fields to $T\mathcal{F}_\pi$ homotopic through taut foliations to $\mathcal{F}_\pi$?
\end{pb}

\subsection{Sutured manifold decompositions}\label{section: suture}

One way to construct taut foliations and tight contact structures on a given $3$-manifold $M$ is to proceed as follows:
\be
\item[(i)] Decompose $M$ into basic pieces.
\item[(ii)] Starting from model foliations/contact structures on these basic pieces, inductively glue them together to construct foliations/contact structures on the desired manifold $M$.
\ee
It is nontrivial to ensure that at each step the gluing is compatible with the foliation/contact structure, and even more difficult in the contact case to ensure that tightness is preserved.

In this subsection we discuss decomposing $M$ into basic pieces. A suitable decomposition, called a {\em sutured manifold decomposition}, was discovered by Gabai \cite{Ga}.

\begin{defn}
A {\em sutured $3$-manifold} is a triple $(M,\Gamma ,U(\Gamma))$, where $M$ is a compact $3$-manifold with corners, $\Gamma$ is an oriented $1$-manifold in $\partial M$ called the {\em suture}, and $U(\Gamma)=[-1,0]\times [-1,1] \times \Gamma$ is a neighborhood of $\Gamma=\{(0,0)\}\times \Gamma$ in $M$ with coordinates $(\tau,t)\in [-1,0]\times [-1,1]$, such that the following hold:
\begin{itemize}
\item $U(\Gamma)\cap \partial M =(\{0 \} \times [-1,1]\times \Gamma) \cup ([-1,0]\times \{ -1,1\} \times \Gamma)$.
\item $\partial M -(\{ 0\} \times (-1,1) \times \Gamma)$ is the disjoint union of two submanifolds which we call $R_- (\Gamma )$ and $R_+(\Gamma )$, where the orientation of $\partial M$ agrees with that of $R_+ (\Gamma)$ and is opposite that of $R_- (\Gamma )$, and the orientation of $\Gamma$ agrees with the boundary orientation of $R_\pm(\Gamma)$.
\item The corners of $M$ are precisely $\{ 0\} \times \{ \pm 1\} \times \Gamma$.
\end{itemize}
We refer to $A(\Gamma)=\{0 \} \times [-1,1]\times \Gamma$ as the {\em vertical annular neighborhood of $\Gamma$} in $\partial M$.
\end{defn}

Our definition is slightly different from Gabai's original one from \cite{Ga}: we introduced the neighborhoods $U(\Gamma)$ and use smooth manifolds with corners instead of ones with boundary. We often suppress the data of the neighborhood $U(\Gamma)$, since it is usually understood.

\begin{defn}
A {\em product sutured manifold} is a sutured manifold of the form
$$(S\times [-1,1], \Gamma=\partial S\times \{0\}),$$
where $S$ is a compact oriented surface, each of whose components has nonempty boundary. Here $A(\Gamma)=\partial S\times [-1,1]$.
\end{defn}

\begin{defn}
A sutured $3$-manifold $(M,\Gamma ,U(\Gamma))$ is {\em taut} if $M$ is irreducible and $R_+(\Gamma )$ and $R_-(\Gamma )$ have minimal {\em Thurston norm}, i.e., the sum of the Euler characteristics of its nonspherical components, amongst properly embedded surfaces realizing the same homology class in $H_2 (M,\Gamma)$. It is {\em balanced} if $\chi(R_+)=\chi(R_-)$, $M$ has no closed components, and $\pi_0(A(\Gamma))\to \pi_0(\bdry M)$ is surjective.
\end{defn}

We now describe how to apply a {\em sutured manifold decomposition} of a sutured manifold $(M,\Gamma )$ into $(M',\Gamma ')$ along a surface $S$. In what follows $N(B)$ denotes a sufficiently small tubular neighborhood of $B$.

Let $S\subset (M,\Gamma)$ be an oriented properly embedded surface such that $S\pitchfork A(\Gamma), \bdry A(\Gamma)$ and:
\begin{itemize}
\item each arc component $c$ of $S\cap A(\Gamma)$ is nonseparating in $A(\Gamma)$;
\item each closed component $c$ of $S\cap A(\Gamma)$ is homologous to $\Gamma\cap A$, where $A$ is the component of $A(\Gamma)$ containing $c$;
\item no component of $S$ is a disk with boundary in $R_\pm(\Gamma )$; and
\item no component of $\partial S$ bounds a disk in $R_\pm(\Gamma )$.
\end{itemize}

Let $M'=M\setminus \op{int} (N(S))$ and let $S'_+$ (resp.\ $S'_-$) be the portion of $\partial N(S)\cap \partial M'$ where the orientation induced from $S$ agrees with (resp.\ is opposite to) the boundary orientation on $\bdry M'$.


We then set
\begin{gather*}
A(\Gamma'):=(A(\Gamma) \cap M') \cup N(S'_+ \cap R_- (\Gamma ))
\cup N(S'_- \cap R_+ (\Gamma )),\\
R_+ (\Gamma '):=((R_+ (\Gamma )\cap M')\cup S'_+ ) -\op{int} (A(\Gamma')),\\
R_- (\Gamma '):=((R_- (\Gamma )\cap M')\cup S'_- ) -\op{int} (A(\Gamma')),
\end{gather*}
where we smooth and introduce corners as appropriate.

A {\em sutured manifold hierarchy} is a sequence of such decompositions
$$(M,\Gamma)=(M_0 ,\Gamma_0 )\stackrel {S_0}\rightsquigarrow (M_1,\Gamma_1) \stackrel{S_1} \rightsquigarrow\dots\stackrel{S_{n-1}}\rightsquigarrow (M_n,\Gamma_n)$$
along $\pi_1$-injective surfaces $S_i\subset M_i$, $i=1,\dots,n-1$, such that $(M_n,\Gamma_n)$ is a product sutured manifold. Gabai \cite{Ga} showed that:

\begin{thm} \label{thm: Gabai sutured hierarchy}
A taut balanced sutured manifold admits a sutured manifold hierarchy.
\end{thm}

\subsection{Construction of hypertight contact structures}

Next we discuss the construction step in the contact case.

\begin{defn}
A {\em sutured contact manifold} $(M,\Gamma,U(\Gamma),\xi)$ is a sutured $3$-manifold $(M,\Gamma,U(\Gamma))$ together with a contact form $\lambda$ for $\xi$ such that:
\begin{itemize}
\item[(i)] the Reeb vector field $R_\lambda$ for $\lambda$ is positively transverse to $R_+ (\Gamma)$ and negatively transverse to $R_-(\Gamma)$;
\item[(ii)] $\lambda =Cdt+\beta$ in $U(\Gamma)$, where $\beta$ is independent of $t$. In particular, $R_\lambda =\frac{1}{C} \partial_t$ in $U(\Gamma)$.
\end{itemize}
A contact form $\lambda$ satisfying (i) and (ii) is {\em adapted to $(M,\Gamma,U(\Gamma))$.}
\end{defn}

The initial step in the construction is a {\em product sutured contact manifold}, i.e., a product sutured manifold $(S\times [-1,1], \Gamma=\partial S\times \{0\})$ with contact form $\lambda= dt+ \beta$, where $d\beta$ is an area form on $S$.

The reverse process of a sutured manifold decomposition is called ``sutured manifold gluing".
By ``going back up'' the sutured manifold hierarchy and gluing carefully, we obtain the following:

\begin{thm}[\cite{CH}] \label{thm: hypertight previous version}
If $(M,\Gamma)$ is a taut balanced sutured manifold, then it carries an adapted hypertight contact form.
\end{thm}

In the rest of this subsection we upgrade Theorem~\ref{thm: hypertight previous version} to:

\begin{thm}\label{thm: hypertight}
If $M$ is a closed, oriented, connected, irreducible $3$-manifold with $H_2(M;\Z)\neq 0$, then $M$ carries a hypertight contact structure.
\end{thm}

Let $S$ be an embedded genus-minimizing surface in a closed $M$ such that $0\not=[S]\in H_2(M;\Z)$. The first step of a sutured hierarchy starting with $(M,\Gamma=\emptyset)$ yields the sutured manifold $(M_1=M\setminus \op{int}(N(S)),\Gamma_1=\emptyset)$ with $R_\pm=S_\pm$, which is taut but not balanced since $\pi_0(A(\Gamma_1))\to \pi_0(\bdry M_1)$ is not surjective.  In order to remedy this we choose nonseparating simple closed curves $\delta_+$ and $\delta_-$ on $R_\pm \subset \bdry M_1$ and take
\begin{gather*}
R_+':= (R_+ - N(\delta_+)) \cup N(\delta_-),\\
R_-':= (R_- - N(\delta_-))\cup N(\delta_+),\\
\Gamma'_1:= \bdry R_+'=\bdry R_-'.
\end{gather*}

The need to make this modification is a reflection of the fact that a Reeb vector field cannot be made transverse to a closed surface.  This creates some difficulties; in particular we risk losing control of the dynamics of the Reeb vector field. We have a way around this and still retain control of the dynamics by making the Reeb vector field transverse to a branched surface.


At this point we recall the following definition; see \cite{Li} for an explanation of the terms:

\begin{defn}[\cite{Li}]
A branched surface is {\em laminar} if the horizontal boundary of fibered neighborhood of the branched surface is incompressible, there is no monogon, there is no Reeb component, and there is no sink disk.
\end{defn}

It was shown in \cite{Li} that a laminar branched surface carries an essential lamination, which can be thought of as a simultaneous generalization of an incompressible surface and a taut foliation.

The key fact about essential laminations that we need is \cite[Theorem~1(d)]{GO}:

\begin{thm}[Gabai-Oertel \cite{GO}] \label{thm: Gabai-Oertel}
Let $\mathcal{L}$ be an essential lamination on $M$. If $\gamma$ is a closed curve which is efficient for $\mathcal{L}$ and $\gamma\cap \mathcal{L}\not=\emptyset$, then $\gamma$ is not contractible in $M$.
\end{thm}

Here a curve $\gamma$ is {\em efficient for $\mathcal{L}$} if no arc $c$ of $\gamma\cap (M-\mathcal{L})$ (or rather its closure $\overline{c}$) can be pushed into a leaf of $\mathcal{L}$ relative to its endpoints.

As a warm-up we prove the following:

\begin{thm}\label{thm: fibration}
If $M$ is a closed $3$-manifold which admits a fibration $\pi: M^3\to S^1$ with fibers of genus $g\geq 1$, then it carries a hypertight contact structure. When the monodromy is pseudo-Anosov and $g>1$, then the unique tight contact structure $\xi_\pi$ homotopic to the tangent plane of the fibers is hypertight.
\end{thm}

\begin{proof}
The case of $g=1$ is straightforward and was observed by Giroux: there exists a contact form such that the characteristic foliation on each torus fiber is linear and the Reeb vector field is tangent to the fibers and ``orthogonal" to the characteristic foliation.

We consider the $g\geq 2$ case. We show that $\xi_\pi$ has a contact $1$-form whose Reeb vector field is transverse to a laminar branched surface $\B$ and has no contractible periodic orbit in $M\setminus N(\B)$.

\begin{lemma}\label{lemma: construction}
Let $T$ be a compact, connected, oriented surface of genus $>1$ with two boundary components $\gamma_1$ and $\gamma_2$ and let $\delta_-,\delta_+ \subset T$ be two nonseparating oriented simple closed curves.  Then there exists a contact form $\alpha$ on $T\times [-1,1]$ with the following properties:
\begin{enumerate}
\item The Reeb vector field $R$ of $\alpha$ is positively transverse to $T\times \{ t\}$, $t\in [-1,1]$, and is tangent to the interval fibers of $(\partial T)\times [-1,1]$.
\item $\xi =\ker \alpha$ is positively transverse to $\gamma_i \times \{ t\}$ for $t\in [-1,1]$ and $i=1,2$, and to $\delta_- \times \{ -1\}$ and $\delta_+ \times \{ 1\}$.
\item $\int_{\gamma_1 \times \{ t\}} \alpha =\int_{\gamma_2 \times \{ t\}} \alpha =\int_{\delta_- \times \{ -1\}} \alpha =\int_{\delta_+ \times \{ 1\}} \alpha =1$.
\end{enumerate}
\end{lemma}

\begin{proof}[Proof of Lemma~\ref{lemma: construction}]
The contact structure will be a modification of the $[-1,1]$-invariant contact structure on $T\times[-1,1]$.
By the Flexibility Lemma \cite[Lemma~5.1]{CH}, it suffices to construct a contact form $\alpha$ near $T\times \{ \pm 1\}$ that satisfies the properties of Lemma~\ref{lemma: construction}.

We perform this construction near $T=T\times\{1\}$ for an arbitrary nonseparating simple closed curve $\delta=\delta_+$; situation for $T\times\{-1\}$ and $\delta_-$ is analogous. As all nonseparating simple closed curves in $T$ are diffeomorphic by a diffeomorphism of $T$ which is the identity near $\bdry T$, it suffices to do the construction for some $\delta$.

Let $\F$ be an oriented singular Morse-Smale foliation on $T$ such that:
\begin{itemize}
\item its singularities are all positive;
\item $\F$ has no periodic orbits; and
\item $\F$ exits from $T$ along its boundary.
\end{itemize}
Then there exists a $1$-form $\lambda$ such that $\F=\ker \lambda$ and $d\lambda$ is an area form on $T$.
The form $dt+\lambda$ on $T\times \R_t$ is contact, its Reeb vector field is $\partial_t$ and the characteristic foliation on $T=T\times \{0\}$ agrees with $\F$. By attaching a cylindrical collar $\bdry T\times[0,1]$ to $T$ along $\bdry T=\bdry T\times\{0\}$ and extending $\lambda$, we can increase the $\lambda$-lengths of each component of the boundary at will and thus assume that they are both equal to a certain value $\kappa\gg 0$.

There is quite a bit of flexibility to choose $\F$; in particular we may choose $\F$ so that $\F$ is transverse to a (necessarily orientable) train track $\mathcal{T}=c_0\cup c_1$, where $c_0$ is a smooth homologically nontrivial simple closed curve and $c_1$ is an arc with both endpoints on $c_0$, approaching $c_0$ from opposite sides. The train track $\mathcal{T}$ carries infinitely homologically nontrivial simple closed curves $\zeta_1,\zeta_2,\dots$; they are all transverse to $\F$ and some $\zeta_j$ has $\lambda$-length greater than $\kappa$.  Now we can enlarge $T$ by attaching another collar so that the $\lambda$-length of each boundary component agrees with the $\lambda$-length of $\delta:=\zeta_j$.  The normalization (3) is obtained by rescaling.
\end{proof}

We state without proof a classical result on the connectedness of the curve complex:

\begin{fact}\label{lemma: curves}
For every pair of nonseparating simple closed curves $\delta_-,\delta_+$ on a closed oriented surface $S$, there exists a sequence of nonseparating simple closed curves $\delta_0,\dots,\delta_n$ such that $\delta_i \cap \delta_{i+1} =\emptyset$, $\delta_0 =\delta_-$, and $\delta_n=\delta_+$.
\end{fact}

Now let $S$ be a fiber of the fibration $\pi:M\to S^1$ and let $\phi: S\stackrel\sim\to S$ be the monodromy. Let $\delta_0,\dots,\delta_n$ be the sequence of nonseparating oriented simple closed curves on $S$ given by Fact~\ref{lemma: curves} for $\delta_0=\phi(\delta_+)$ and $\delta_n=\delta_+$. We consider the surfaces $S_i =S\times \{ \tfrac{i}{n}\}$, $i=0,\dots,n$, in $S\times [0,1]$.

We briefly comment on orientations of $\delta_i$:  Choose an orientation for $\delta_n=\delta_+$.  Then the orientation for $\delta_0=\phi(\delta_+)$ is determined.  Choose orientations on $\delta_1,\dots,\delta_{n-1}$ arbitrarily.

Let $S_i' =S_i \setminus U_i$, where $U_i=N(\delta_i \times \{ \tfrac{i}{n}\}) \subset S_i$, and let $\gamma_i^1$ and $\gamma_i^2$ be the two components of $\partial S_i'$. Suppose the orientation on $\gamma_i^1$ agrees with that of $\delta_i$.

We apply Lemma \ref{lemma: construction} to a small neighborhood $S_i' \times [-\epsilon,\epsilon]$ of $S_i'$ to obtain a contact form
$\alpha_i$ on $S_i' \times [-\epsilon,\epsilon]$, for which the curves $\delta_{i-1}\times\{-\epsilon\}$ in $S_i' \times \{-\epsilon\}$ and $-\delta_{i+1}\times\{\epsilon\}$ (note the orientation reversal) in $S_i' \times \{\epsilon\}$ are $\alpha_i$-positively transverse and have $\alpha_i$-length $1$, and $\gamma_i^1\times\{\pm\epsilon\}$ and $\gamma_i^2\times \{ \pm \epsilon\}$ have $\alpha_i$-length $1$.

We then identify a neighborhood of $\gamma_i^2 \times \{-\epsilon\}\subset S_i'\times\{-\epsilon\}$ with a neighborhood of $-\delta_i \times \{\epsilon\}\subset S_{i-1}' \times \{\epsilon\}$ and a neighborhood of $\gamma_i^1 \times \{ \epsilon\}\subset S'_i\times\{\epsilon\}$ with a neighborhood of $\delta_i \times \{ -\epsilon\} \subset S_{i+1}' \times \{ -\epsilon\}$, in a manner compatible with the forms $\alpha_i$.

After smoothing we obtain a fibered neighborhood $N(\B)\subset M$ of a branched surface $\B$ whose vertical fibers are tangent to the Reeb vector field $R$ of a contact form $\alpha$. One easily verifies that $\B$ is a laminar branched surface (for example, sink disks do not exist by our choice of curves $\delta_i$).

It remains to extend $\alpha$ to the exterior of $N(\B)$. By construction, $M\setminus \op{int} (N(\B ))$ is diffeomorphic to $\Sigma \times [0,1]$, where $\Sigma$ is a compact oriented surface with boundary. The Reeb vector field is positively transverse to $\Sigma \times \{ 0,1\}$ for a suitable orientation of $\Sigma$ and tangent to $(\partial \Sigma )\times [0,1]$.  Hence, using the Flexibility Lemma, we can extend the contact form $\alpha$ and the Reeb vector field $R$ to $\Sigma_0 \times [0,1]$ so that the Reeb vector field gives a fibration by intervals. Let us write $\xi=\ker \alpha$.




\begin{claim}
The Reeb field $R$ has no contractible periodic orbit, i.e., $\xi$ is hypertight.
\end{claim}

\begin{proof}
Every periodic orbit $\zeta$ of $R$ must intersect $\B$. Since $\B$ is laminar, $\zeta$ is also transverse to an essential lamination $\mathcal{L}$ carried by $\B$.  The curve $\zeta$ is efficient with respect to $\mathcal{L}$, for example since $\mathcal{L}$ is orientable.  Theorem~\ref{thm: Gabai-Oertel} then implies that $\zeta$ is not contractible.
\end{proof}


One easily checks that $\langle e(\xi),[F]\rangle =-\chi (F)$. Since $\xi$ is hypertight by the claim, it is tight, and is therefore contact isotopic to $\xi_\pi$ by Theorem~\ref{thm: unique tight for fibered}.  This completes the proof of Theorem~\ref{thm: fibration}.
\end{proof}

This strategy extends to control the dynamics in the last gluing of a sutured manifold hierarchy in the general case:

\begin{proof}[Proof of Theorem~\ref{thm: hypertight}]
Let $S$ be an embedded genus-minimizing surface in $M$ with $[S]\not=0$. We first cut $M$ along $S$ to obtain $(M_1,\Gamma_1)$ with boundary $R_+(\Gamma_1)=S_+$ and $R_-(\Gamma_1)=S_-$ and $\Gamma_1=\emptyset$. Next we cut $M_1$ along the surface $S_1\subset M_1$, with $\partial S_1$ intersecting $S_+$ and $S_-$ along multicurves $\delta_+$ and $\delta_-$, respectively, to obtain $(M_2,\Gamma_2)$ such that $A(\Gamma_2)$ corresponds to thickenings of $\delta_\pm$.  We may take each of $\delta_+$ and $\delta_-$ to be a union of parallel oriented nonseparating simple closed curves, since Gabai's construction of the sutured manifold hierarchy guarantees such ``well-grooming" of  $\delta_\pm=S_1\cap S_\pm$.  Theorem~\ref{thm: hypertight previous version} constructs a hypertight contact form on $(M_2,\Gamma_2)$.

The rest of the proof is similar to that of Theorem~\ref{thm: fibration}:  We take a sequence $\delta_0,\dots,\delta_n$ of curves in $S$ from $\delta_0=\delta_+$ to $\delta_n=\delta_-$  furnished by Fact~\ref{lemma: curves}; here $\delta_+$ and $\delta_-$ are unions of parallel curves and each $\delta_i$, $i=1,\dots,n-1$ is connected.  Following the proof of Theorem~\ref{thm: fibration}, we construct contact forms $\alpha_i$ on $S_i'\times[-\epsilon,\epsilon]$ for $i=1,\dots,n-1$ and glue the pieces $S_i'\times[-\epsilon,\epsilon]$, $i=1,\dots,n-1$, in the same way.  We also glue $M_2$ and $S_1'\times[-\epsilon,\epsilon]$ so that a neighborhood of $\delta_0\subset R_+(\Gamma_2)$ is identified with a neighborhood of $\delta_0\times\{-\epsilon\}\subset S'_1\times\{-\epsilon\}$ and a neighborhood of $\gamma_1^2\times\{-\epsilon\}\subset S_1'\times\{-\epsilon\}$ is identified with a neighborhood of $-\delta_1\subset R_+(\Gamma_2)$ and analogously glue $M_2$ and $S_{n-1}'\times[-\epsilon,\epsilon]$. Note that since $\delta_+$ (also $\delta_-$) is not necessarily connected we need to make the appropriate adjustments to Lemma~\ref{lemma: construction}(3).  We then extend the contact form to the exterior of $M_2\cup (\cup_{i=1}^{n-1} S_i'\times[-\epsilon,\epsilon])/\sim$, which is of the form $\Sigma\times[0,1]$.  All the periodic orbits must efficiently intersect a laminar branched surface and hence are not contractible by Theorem~\ref{thm: Gabai-Oertel}.
\end{proof}

\begin{pb}\label{pb: Haken}
Does every Haken manifold carry a tight contact structure?
\end{pb}

\begin{pb}\label{pb: hypertight vs universally tight}
Is every universally tight contact structure on a $3$-manifold with universal cover $\R^3$ hypertight?
\end{pb}

\section{Invariants of contact manifolds}\label{section: invariant}

\subsection{Sutured contact homology and sutured ECH}

For more details see~\cite{EGH}, \cite{Hu1,Hu2,Hu3}, \cite{CGHH} for contact homology, ECH, and sutured versions of these theories, respectively.

There are two main invariants of a contact manifold $(M,\xi)$, obtained as the homology of chain complexes generated by finite products $\bs\gamma=\prod_i \gamma_i$ of periodic orbits of a Reeb vector field $R_\lambda$, $\xi=\ker \lambda$, subject to some conditions and whose differentials count (Fredholm or ECH) index $1$ $J$-holomorphic curves that limit to cylinders over collections of periodic orbits at the $s\to\pm\infty$ ends in the symplectization $(\R_s\times M,d(e^s \lambda))$ for some ``adapted'' almost complex structure $J$, modulo translation in the $s$-direction. In this text we use $\mathbb{F}=\Z/2$-coefficients.

The first invariant {\em contact homology} is sensitive to the particular choice of contact structure. In the case of a hypertight contact form $\lambda$, there is a simpler version where the chain complex can be taken to be the free $\mathbb{F}$-vector space generated by the (good) periodic orbits of $R_\lambda$ and the differential to be a count of $J$-holomorphic cylinders; this yields the {\em cylindrical contact homology of $(M,\lambda)$}.

The second invariant {\em embedded contact homology} (ECH) is a topological invariant of the ambient manifold $M$ and is isomorphic to the Heegaard Floer homology of $-M$, i.e., $M$ with the orientation reversed. In particular, the contact class $c(M,\xi)$, associated to the contact structure $(M,\xi)$ in the Heegaard Floer hat group $\widehat{HF}(-M)$, is defined in the ECH hat group $\widehat{ECH}(M)$ to be the homology class of the empty set of orbits.

In the case of a sutured contact manifold $(M,\Gamma,\xi=\ker \lambda)$, the hypothesis that $\lambda$ is a contact form adapted to $(M,\Gamma)$, together with the choice of a {\em tailored} almost complex structure $J$ on the symplectization from \cite{CGHH}, prevents a sequence of holomorphic curves from exiting the symplectization $\R\times M$ along its boundary $\R\times \bdry M$ and allows us to extend the definition of contact homology and ECH to sutured manifolds \cite{CGHH}. The resulting groups are the {\em sutured contact homology} and {\em sutured ECH} groups and are denoted by $HC(M,\Gamma,\xi)$ and $ECH(M,\Gamma,\xi)$.

The following gluing theorem holds for both:

\begin{thm}[\cite{CGHH}]\label{thm: injection}
If $(M',\Gamma',\xi')$ is obtained from $(M,\Gamma ,\xi)$ by a sutured gluing, then there are injections
\begin{gather*}
ECH(M,\Gamma,\xi)\hookrightarrow ECH(M',\Gamma',\xi'),\\
HC(M,\Gamma,\xi)\hookrightarrow HC(M',\Gamma',\xi').
\end{gather*}
\end{thm}

Let $M$ be a closed manifold and $K\subset M$ be a knot. We define the {\em knot ECH} of $(M,K)$ as follows: Let $N(K)$ be a small tubular neighborhood of $K$, $\Gamma$ two parallel meridian sutures on $\bdry N(K)$, and $\xi$ a contact structure adapted to $(M\setminus\op{int}(N(K)),\Gamma)$.  Then the knot ECH is
$$ECH(M\setminus\op{int}(N(K)),\Gamma,\xi),$$
defined in analogy with Juh\'asz' observation that the knot Heegaard Floer homology group $\widehat{HFK}(M,K)$ satisfies:
$$\widehat{HFK}(M,K)\simeq HF(M\setminus\op{int}(N(K)),\Gamma).$$
Here we denote the sutured Heegaard Floer homology of a sutured manifold $\mathcal{S}$ (in a slightly nonstandard manner) by $HF(\mathcal{S})$.

Let $(M,\xi)$ be a closed contact manifold and $\Lambda\subset(M,\xi)$ a Legendrian knot. We define the contact homology Legendrian invariant of $\Lambda$ to be
$$HC(M\setminus \op{int}(N(\Lambda)),\Gamma_\Lambda,\xi),$$
where $N(\Lambda)$ is the standard tubular neighborhood of $\Lambda$ and $\Gamma_\Lambda$ has two parallel sutures whose slope is induced by the framing along $\Lambda$ given by $\xi$.

\begin{pb}
Compute this Legendrian knot invariant on significant examples.
\end{pb}

\subsection{ECH and Heegaard Floer homology for a sutured manifold}

In work in progress with Ghiggini and Spano \cite{CGHS}, we study the extension of the isomorphism between ECH and Heegaard Floer homology from~\cite{CGH0, CGH1,CGH2,CGH3} to sutured manifolds.

The ECH group $ECH(M,\Gamma,\xi)$ admits a decomposition into homology classes $A\in H_1(M;\Z)$ of ``orbit sets'' $\bs\gamma$ as follows:
$$ECH(M,\Gamma,\xi)=\oplus_{A\in H_1(M;\Z)} ECH(M,\Gamma,\xi,A).$$
The sutured Heegaard Floer homology group $HF(M,\Gamma)$ admits a decomposition into relative Spin$^c$-structures which form an affine space over $H^2(M,\bdry M;\Z)$:
$$HF(M,\Gamma)=\oplus_{{\frak s}\in H^2(M,\bdry M;\Z)} HF(M,\Gamma,{\frak s}).$$

We expect to prove the following:

\begin{conj}\label{thm: sutured manifold}
If $(M,\Gamma,\xi =\ker \alpha )$ is a sutured contact manifold, then
$$ECH(M,\Gamma,\xi,A)\simeq HF(-M,-\Gamma,{\frak s}_\xi +PD(A)).$$
\end{conj}

{\em In the rest of this subsection we will refer to consequences of the conjecture and its expected proof as ``corollaries".}

\subsubsection{Case of a knot}

The sutured isomorphism specializes to an isomorphism between knot ECH and knot Heegaard Floer homology.

\begin{cor}\label{thm: general knot}
Let $M$ be a closed manifold, $K\subset M$ a knot, $N(K)$ its tubular neighborhood, and $\Gamma$ two parallel meridian sutures on $\bdry N(K)$. Also let $\xi$ be a contact structure adapted to $(M\setminus\op{int}(N(K)),\Gamma)$. Then
\begin{align*}
ECH(M\setminus\op{int}(N(K)),\Gamma,\xi)& \simeq HF(-(M\setminus \op{int}(N(K))),-\Gamma)\\
& \simeq \widehat{HFK}(-M,-K).
\end{align*}
\end{cor}

Suppose that $K\subset M$ is nullhomologous.  Fix a minimum genus Seifert surface $S$ for $K$ with genus $g$.  Then we can decompose $ECH(M\setminus\op{int}(N(K)),\Gamma,\xi)$ according to the homological intersection number $\iota(\bs\gamma)=\langle \bs\gamma,S\rangle$.  Similarly,
$$\widehat{HFK}(-M,-K)=\oplus_{j=-g}^{g} \widehat{HFK}(-M,-K,j)$$
where the decomposition is according to the relative $\op{Spin}^s$-structure, normalized so that its graded Euler characteristic $\sum_j \op{rk}\widehat{HFK}(-M,-K,j)\cdot t^j$ agrees with the Alexander polynomial of $K$. Corollary~\ref{thm: general knot} can be written more precisely as:
\begin{equation} \label{eqn: more precise}
ECH(M\setminus\op{int}(N(K)),\Gamma,\xi,\iota)\simeq \widehat{HFK}(-M,-K,g-\iota).
\end{equation}

By Equation~\eqref{eqn: more precise}, the ECH of a knot is supported in grading $0\leq \iota\leq 2g$. Combined with the work of Ghiggini~\cite{Gh} and Ni~\cite{Ni},
\begin{equation}
ECH(M\setminus\op{int}(N(K)),\Gamma,\xi,\iota=0)\simeq \mathbb{F}\langle\emptyset\rangle
\end{equation}
if and only if the knot is fibered.

\subsubsection{Relationship to symplectic Floer homology}

If $K\subset M$ is fibered, then $ECH(M\setminus\op{int}(N(K)),\Gamma,\xi,\iota=1)$ is the symplectic Floer homology of a special representative of the monodromy.

More precisely, let $K\subset M$ be a fibered knot with fiber $S$ and monodromy $\hh: S\stackrel\sim\to S$ such that $\hh|_{\bdry S}=\op{id}$. Let $N(\partial S)= \partial S \times [0,1]$ be a neighborhood of $\partial S=\bdry S\times \{0\}$ in $S$. We isotop $\hh$ to $\hh'$ on $S$ so that:
\be
\item during the isotopy $\bdry S\times\{0\}$ is rotated at most $\epsilon$, and
\item $\hh'$ is a rotation by an angle $\phi(r)$ on $\partial S\times \{r\}$ (here we are viewing $\bdry S\simeq S^1$ using the boundary orientation), where $\phi:[0,1]\to \R$ satisfies $\phi(0)=-\epsilon$, $\phi(\tfrac{1}{2})=0$, $\phi(1)=\epsilon$, and $\phi'(r)>0$.
\ee
We are also assuming that $\epsilon>0$ is sufficiently small so that the only periodic points of period $1$ in $N(\partial S)$ are along $\{r=\tfrac{1}{2}\}$.
We then perturb $\hh'$ to a flux zero area-preserving map $\hh''$ with respect to some area form $\omega$ so that all the fixed points are nondegenerate. In particular, two fixed points appear along $\{ r=\tfrac{1}{2}\}$, an elliptic one $e$ and a hyperbolic one $h$.

Let $SF^{\sharp}(S,\hh'')$ be the symplectic Floer homology of $\hh''$, whose chain complex is generated by all its fixed points but $e$.  By the methods of Hutchings and Sullivan \cite{HS} combined with \cite[Theorem~10.3.2]{CGHH},
\begin{equation}\label{eqn: next to top}
ECH(M\setminus\op{int}(N(K)),\Gamma,\xi,\iota=1)\simeq SF^{\sharp}(S,\hh'').
\end{equation}
Combining Equations~\eqref{eqn: more precise} and ~\eqref{eqn: next to top}, we obtain:

\begin{cor}\label{thm: monodromy}
$\widehat{HFK}(-M,-K,g-1)\simeq SF^{\sharp}(S,\hh'')$.
\end{cor}

This result was implicitly conjectured by the two authors and Ghiggini as a combination of \cite[Conjecture~1.5]{CGHH} and
\cite[Theorem~10.3.2]{CGH0}, precisely formulated in the thesis \cite{Sp}, and first proven using a different argument in \cite{GS}. Recently, Kotelskiy \cite{Ko} reformulated this conjecture in the language of bordered Floer homology and gave some evidence for it.

\subsubsection{Dynamical characterization of product sutured manifolds}

Using Conjecture~\ref{thm: sutured manifold}, one can transfer results from the Heegaard Floer side to the ECH side.  The following is obtained by transferring the corresponding Heegaard Floer characterization (due to Juh\'asz \cite{Ju}) to ECH.

\begin{cor}\label{cor: product1}
The sutured contact manifold $(M,\Gamma, \xi)$ satisfies
$$ECH(M,\Gamma,\xi)\simeq\mathbb{F}$$
if and only if $(M,\Gamma)$ is a product sutured manifold.
\end{cor}

Here is another formulation, which answers a question of John Pardon and can be thought of as an extension of the Weinstein conjecture \cite{T1} to sutured contact $3$-manifolds.

\begin{cor}\label{cor: product2}
If $(M,\Gamma, \xi)$ is a sutured contact manifold with adapted contact form $\lambda$ whose Reeb vector field $R_\lambda$ has no periodic orbit, then $(M,\Gamma,\xi)$ is a tight product sutured contact manifold $(S\times [-1,1],\Gamma=\bdry S\times\{0\},\xi)$, where $\xi$ is $[-1,1]$-invariant.  If $S$ is planar in addition, then every orbit of $R_\lambda$ flows from $S\times \{-1\}$ to $S\times \{1\}$; in particular $R_\lambda$ has no trapped orbits.
\end{cor}

\begin{pb}
Can one prove that if there is no periodic orbit in $S\times [0,1]$, then there is also no trapped orbit even when $S$ is not planar?
\end{pb}

In the higher-dimensional case, there can be trapped orbits without periodic ones, as proven by Geiges, R\"ottgen and Zehmisch~\cite{GRZ}.

Recall the {\em depth} of a sutured manifold is the minimum number of steps in a sutured hierarchy needed to get to a product sutured manifold. We have the following dynamical characterization of depth $k$ sutured manifolds:

\begin{cor}\label{thm: depth}
If $(M,\Gamma ,\xi =\ker \lambda)$ is a taut sutured contact manifold of depth greater than $2k$ with $H_2 (M)=0$ and if $R_\lambda$ is nondegenerate and has no elliptic orbit, then it has at least $k+1$ hyperbolic orbits.
\end{cor}

\begin{proof}
Under the hypothesis of the theorem, Juh\'asz \cite[Theorem~4]{Ju2} shows that $$\op{rk} HF(-M,-\Gamma) \geq 2^{k+1}.$$
By Conjecture~\ref{thm: sutured manifold}, the ECH chain complex must have rank $\geq 2^{k+1}$. When there are no elliptic orbits, this implies the existence of at least $k+1$ hyperbolic orbits for $R_\lambda$.
\end{proof}

Note that every Reeb vector field can be perturbed so that it only has hyperbolic orbits up to a certain action threshold $L$ \cite[Theorem~2.5.2]{CGH1} and the number of hyperbolic orbits seems to go to $\infty$ as $L\to \infty$ whenever there is an elliptic orbit to start with.

\subsubsection{Nonvanishing of contact invariants}

One of the central tools in recent years has been the following:

\begin{thm}[Ozsv\'ath-Szab\'o~\cite{OSz4}]\label{thm: contact class}
If $(M,\xi)$ is a symplectically semi-fillable contact structure, then the contact invariant $c(\xi)$ in $\widehat{HF}(M;\mathbb{F}[H_2(M;\Z)])$ with respect to ``twisted coefficients'' is nonzero. In particular, it is the case when $\xi$ is the deformation of a taut foliation.
\end{thm}

This result contrasts with the fact that $c(\xi)=0$ when $\xi$ is overtwisted.


\begin{proof}
We explain how to prove Theorem~\ref{thm: contact class} in the equivalent setting of ECH, {\em assuming the existence of cobordism maps that are defined on the chain level and count $J$-holomorphic curves.}  This is work in progress of Jacob Rooney~\cite{Ro}.

Let $(W,d\beta)$ be a symplectic semi-filling of $(M,\xi)$ such that $\bdry W=(M,\xi)\sqcup (M',\xi')$.
The semi-filling $W$, together with an adapted almost complex structure $J$ on $W$, induces a cobordism map
$$\Phi: ECC(M)\otimes ECC(M')\to ECC(\emptyset)\simeq \mathbb{F},$$
where $ECC$ denotes the ECH chain complex; it satisfies $\Phi\bdry=0$ since $\bdry\Phi=0$. Recall that the contact class $c(\xi)$ is represented by the empty set $\emptyset$ of orbits; $\Phi$ is defined so it satisfies $\Phi(\emptyset\otimes\emptyset)=\emptyset$.  If $c(\xi)=0$, then there exists $x\in ECC(M)$ such that $\bdry x=\emptyset$. Then
$$\Phi\bdry (x\otimes\emptyset)= \Phi(\emptyset\otimes \emptyset)=\emptyset,$$
which contradicts $\Phi\bdry=0$.
\end{proof}


Using the same circle of ideas, Ozsv\'ath-Szab\'o~\cite{OSz4} showed that if $M$ admits a taut foliation, then $\op{rk} \widehat{HF}(M)>|H_1 (M;\Z)|$, i.e., $M$ is not an $L$-space, leading to the following conjecture \cite{BGW, Ju4}:

\begin{conj}
A closed rational homology $3$-sphere $M$ carries a taut foliation if and only if $M$ is not an $L$-space.
\end{conj}

It is also conjectured that the existence of a taut foliation is equivalent to the left-orderability of $\pi_1(M)$ for a rational homology sphere $M$. In one direction, Calegari-Dunfield~\cite[Corollary~7.6]{CD} showed that the existence of a taut foliation on a atoroidal manifold implies the left-orderability of $\pi_1$ of a finite cover of $M$.

\begin{pb}
Find a relationship between contact geometry and the left-orderability of $\pi_1(M)$.
\end{pb}

\begin{pb}
Show that if $M$ is Haken and $\xi$ is a contact structure on $M$, then $\op{rk}(\widehat{ECH} (M,\xi)) \geq 2$. One might want to find tight contact structures with nonvanishing contact class; see Problem \ref{pb: Haken}.
\end{pb}

\subsection{Cylindrical contact homology of sutured manifold}

In this subsection, we prove another contact characterization of product sutured manifolds:

\begin{thm}\label{thm: cylindrical}
A taut balanced sutured manifold $(M,\Gamma)$ is not a product if and only if it carries an adapted hypertight contact form whose cylindrical contact homology has rank $\geq 1$.
\end{thm}

First recall from \cite{CGHH} that when $(M,\Gamma,U(\Gamma),\xi=\ker \lambda)$ is a sutured contact manifold, we can define its {\em vertical completion} by gluing
\begin{itemize}
\item $([1,+\infty)\times R_+(\Gamma),dt+\lambda\vert_{R_+(\Gamma)})$ to $(M,\lambda)$ along
$R_+(\Gamma)=\{1\} \times R_+(\Gamma)$ and
\item $((-\infty,-1]\times R_-(\Gamma),dt+\lambda\vert_{R_-(\Gamma)})$ to $(M,\lambda)$ along $R_-(\Gamma)=\{1\} \times R_-(\Gamma)$.
\end{itemize}

\begin{proof}
If $(M,\Gamma)$ is a taut balanced sutured manifold, then it admits a sutured manifold hierarchy by Gabai~\cite{Ga} and carries a hypertight contact structure by Theorem~\ref{thm: hypertight previous version}; this is obtained from the product sutured contact manifold by successive gluing. By Theorem \ref{thm: injection} we know that, at each gluing step of the hierarchy, the cylindrical contact homology of the previous piece injects into the cylindrical contact homology of the next one.

In order to prove the theorem, it suffices to show that gluing step corresponding to the decomposition $(P',\Gamma')\stackrel{S}\rightsquigarrow (P,\Gamma)$, where $(P,\Gamma)$ is a product and $(P',\Gamma')$ is not a product, produces $(P',\Gamma',\xi')$ with nontrivial cylindrical contact homology.

We first give a brief description of the gluing step:  Let $(P,\Gamma,\xi)$ be a product sutured contact manifold with contact form $dt+\lambda$ where
$$P=[-1,1]_t\times R ~~~\mbox{ and } ~~~ R_\pm(\Gamma)=\{\pm 1\}\times R=\{\pm 1\}\times R_\pm (\Gamma).$$
Let $S_\pm$ be a compact subsurface with corners in $R_\pm (\Gamma)$ such that $\partial S_\pm$ is the concatenation of arcs $a_\pm^1,b_\pm^1,\dots, a_\pm^k,b_\pm^k$ in (cyclic) order, where
$$\op{int}(a_\pm^i) \subset \op{int}(R_\pm (\Gamma))~~~\mbox{ and } ~~~ b_\pm^i \subset \partial R_\pm (\Gamma).$$
Let $\phi :S_+\stackrel\sim \to S_-$ be a diffeomorphism with $\phi(a_+^i) =b_-^i$ and $\phi(b_+^i) =a_-^i$. Then, modulo some adjustments to the contact form that we need to make near $S_+$ and $S_-$, we can glue $(P,\Gamma,\xi)$ along $S_+$ and $S_-$ using $\phi$ and apply smoothing operations to obtain the sutured contact manifold $(P',\Gamma',\xi')$; see \cite{CGHH}.

Another way to construct $(P',\Gamma',\xi')$ is as follows and is described in detail in \cite[Section 4.3, Steps 1-4]{CGHH}:
\be
\item Glue $(P,\Gamma,\xi)$ along $S_+$ and $S_-$ using $\phi$.
\item Glue $([1,3]\times (R_+(\Gamma) \setminus S_+), dt+\lambda\vert_{R_+(\Gamma)})$ to the result of (1) by identifying
\begin{gather*}
\{1\} \times (R_+(\Gamma) \setminus S_+) ~~~\mbox{ and } ~~~ \{1\} \times (R_+(\Gamma) \setminus S_+)\\
[1,3]\times a_+ ~~~ \mbox{ and } ~~~ [-1,1]\times b_-,
\end{gather*}
where we are writing $a_\pm= \cup_{i=1}^k a_\pm ^k$ and $b_\pm =\cup_{i=1}^k b_\pm^k$.
\item Iterate the process to glue fibered pieces
\begin{gather*}
([2n+1,2n+3]\times (R_+(\Gamma) \setminus \op{int}(S_+)), dt+\lambda\vert_{R_+(\Gamma)}),~~~ n\geq 1\\
([-2n-1,-2n+1]\times (R_-(\Gamma) \setminus \op{int}(S_-)), dt+\lambda\vert_{R_-(\Gamma)}),~~~ n\geq 1.
\end{gather*}
\ee

This leads us to the vertical completion $(P'_\infty,\Gamma'_\infty,\xi'_\infty)$ of $(P',\Gamma',\xi')$; see Figure~\ref{figure: spiral}.
\begin{figure}[ht]
\begin{overpic}[width=9cm]{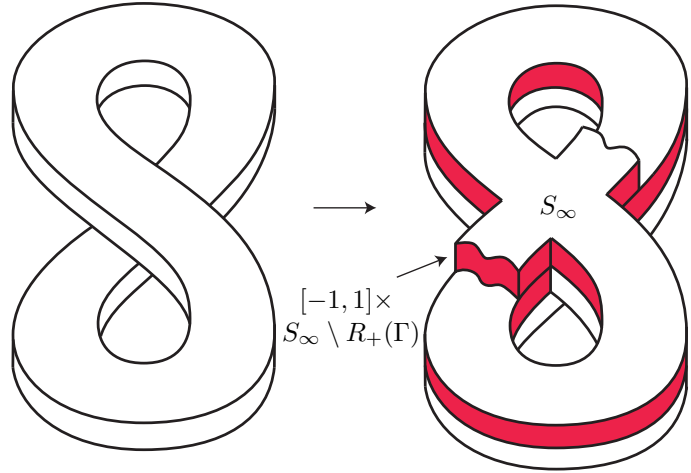}
\put(78,38){\small $S_\infty$}
\put(42.5,24){\small $[-1,1]\times $ }
\put(40,19.2){\small $S_\infty\setminus R_+(\Gamma)$}
\end{overpic}
\caption{The construction of $S_\infty$} \label{figure: spiral}
\end{figure}


The noncompact surface $S_\infty \subset P'_\infty$ given by
\begin{align*}
S_\infty & =(\{1\}\times R) \cup (\cup_{n\geq 1} \{2n+1\} \times  (R_+(\Gamma) \setminus \op{int}(S_+)))\\
& \qquad \qquad \cup ( \cup_{n\geq 1} \{-2n-1\} \times  (R_-(\Gamma) \setminus \op{int}(S_-)))
\end{align*}
divides $(P'_\infty,\Gamma'_\infty,\xi'_\infty)$ into a noncompact sutured manifold $(P_\infty,\Gamma_\infty,\xi_\infty)$, which is an horizontal extension of $(P,\Gamma,\xi)$ by product pieces.  Conversely, $(P'_\infty,\Gamma'_\infty,\xi'_\infty)$ is obtained from  $(P_\infty,\Gamma_\infty,\xi_\infty)$ by gluing $R_+(\Gamma_\infty)\simeq S_\infty$ to $R_-(\Gamma_\infty)\simeq S_\infty$ by a diffeomorphism $\phi_\infty : S_\infty \stackrel\sim\longrightarrow S_\infty$.

An analysis of the situation shows that the surface $S_\infty$ is an {\em end-periodic surface} and $\phi_\infty$ an {\em end-periodic diffeomorphism} of $S_\infty$. We refer to \cite{Fe} for precise definitions. End-periodic diffeomorphisms of end-periodic surfaces have a Thurston-Nielsen theory which was discovered by Handel-Miller, described by Fenley~\cite{Fe}, and developed by Cantwell-Conlon~\cite{CC}.

In our case, the relevant fact is the following:

\begin{fact}
If  $(P'_\infty,\Gamma'_\infty)$ --- or equivalently its compactification $(P',\Gamma')$ --- is not a product, then $\phi_\infty: S_\infty\stackrel\sim\longrightarrow S_\infty$ is isotopic to an end-periodic homeomorphism $\psi_\infty: S_\infty\stackrel\sim\longrightarrow S_\infty$ which has a finite number of periodic points of minimum period, say $k$, and with nonzero total Lefschetz number in their Nielsen classes.
\end{fact}

Briefly, if indeed $(P'_\infty,\Gamma'_\infty)$ is not a product, then $\phi_\infty$ is not isotopic to a translation and there exist $\psi_\infty$ isotopic to $\phi_\infty$ and a pair of geodesic laminations $\Lambda_+,\Lambda_-\subset S_\infty$ with nonempty intersection \cite[Proposition 4.28]{CC} that are preserved by $\psi_\infty$ \cite[Theorem 4.54]{CC}, such that:
\begin{enumerate}
\item The dynamical system
$$\psi_\infty :  \Lambda_+ \cap \Lambda_- \to \Lambda_+ \cap \Lambda_-$$
is conjugate to a two-sided Markov shift of finite type \cite[Theorem 9.2]{CC}.
\item There is a union $\mathcal{P}_+$ (resp.\ $\mathcal{P}_-$) of components of $S_\infty \setminus \Lambda_+$ (resp.\ $S_\infty \setminus \Lambda_-$) such that all the points of $S_\infty \setminus(\Lambda_+\cup \mathcal{P}_+)$ (resp.\ $S_\infty \setminus (\Lambda_-\cup \mathcal{P}_-)$)  belong to the ``negative (resp.\ positive) escaping set" and hence are not periodic points \cite[Lemma 5.15]{CC}.
\end{enumerate}

There exists a compact region of the form $cl(Q_+\cap Q_-)$, where $Q_\pm$ is a component of $\mathcal{P}_\pm$ (this is called a {\em nucleus} \cite[Definition 6.41]{CC}), whose piecewise smooth boundary alternates between arcs of $\Lambda_+$ and arcs of $\Lambda_-$ and which is left invariant by some $\psi^k_\infty$. If $cl(Q_+\cap Q_-)$ is a rectangle, then there is a hyperbolic orbit which is isolated in its Nielsen class.  Otherwise we apply the Thurston-Nielsen decomposition to $(cl(Q_+\cap Q_-),\psi^k_\infty)$. If there is a pseudo-Anosov component, then there are infinitely many Nielsen classes with nonzero total Lefschetz number. If all the components are periodic, at least one component has negative Euler number, again giving rise to a Nielsen class with nonzero total Lefschetz number.

In particular, in every such Nielsen class, $\phi_\infty$ itself has periodic points with nonzero total Lefschetz number. Those periodic points correspond to closed orbits of a Reeb vector field that intersects $S_\infty$ $k$ times. The differential in cylindrical contact homology counts holomorphic cylinders, and hence preserves the splitting of the orbits into Nielsen classes. Moreover, the total Lefschetz number in a Nielsen class corresponds to the Euler characteristic of the corresponding subcomplex, so every Nielsen class with nonzero Lefschetz number gives at least one generator in cylindrical contact homology.

Conversely, if $(M,\Gamma)$ is a product sutured manifold and if a contact structure $\xi$ adapted to $(M,\Gamma)$ has a well-defined cylindrical contact homology, then $\xi$ is tight by Hofer~\cite{Ho}. Since $(M,\Gamma)$ is product disk decomposable, we can normalize the contact structure on a collection of compressing disks and apply Eliashberg's classification of tight contact structures on the remaining ball \cite{El2} to show that there is a unique tight contact structure  adapted to $(M,\Gamma)$, namely the $[-1,1]$-invariant one. Hence $\xi$ has a Reeb vector field without periodic orbits and its cylindrical contact homology is trivial. 
\end{proof}

\begin{rmk}
The above proof allows us to estimate the size of $HC(M,\Gamma,\xi)$ depending on $\phi_\infty$. For example, if $\phi_\infty$ is reducible with a pseudo-Anosov component, then $\op{rk} HC(M,\Gamma,\xi)=\infty$ and the rank of the subspace generated by orbits of action less than $L>0$ grows exponentially with $L$.
\end{rmk}

\begin{pb}
Use this presentation via end-periodic diffeomorphisms of end-periodic surfaces to give a direct proof of Corollary~\ref{cor: product1}, i.e., if $(M,\Gamma)$ is not a product, then $\op{rk} ECH(M,\Gamma,\xi)\geq 2$, where one of the generators is the empty set of orbits. The difference with the previous case is that the ECH differential counts curves with genus that might not preserve Nielsen classes.
\end{pb}

\begin{pb}
Show that a hypertight contact structure on a closed hyperbolic $3$-manifold has infinite rank cylindrical contact homology and that the rank of the subspace of its contact homology generated by orbits of action less than $L>0$ grows exponentially with $L$. See Problem~\ref{pb: fibration}.
\end{pb}

\section{From contact structures to foliations}

As we have already seen, a taut foliation $\F$ without sphere leaves can be approximated by a positive and a negative contact structure $\xi_+$ and $\xi_-$. Both are used to construct a weak filling and thus prove the nonvanishing of the contact invariant.

In \cite{CF}, the first author and Firmo studied the converse construction, based on prior results of Eliashberg-Thurston \cite{ET} and Mitsumatsu \cite{Mi}. Since the contact structures approximate the same foliation, they are $C^0$-close to each other, and in particular they have a common positively transverse vector field.

This motivates the following definition:

\begin{defn}
A {\em contact pair} is a pair $(\xi_+,\xi_-)$ consisting of a positive contact structure $\xi_+$ and a negative contact structure $\xi_-$ such that there exists a nonvanishing vector field $Z$ which is positively transverse to both $\xi_+$ and $\xi_-$.  

A contact pair is {\em tight} if both $\xi_+$ and $\xi_-$ are tight.
\end{defn}

In particular, if $\xi_+$ and $\xi_-$ are everywhere transverse, then $(\xi_+,\xi_-)$ is a contact pair.

Let $(\xi_+,\xi_-)$ be contact pair.  Generically $\xi_+$ and $\xi_-$ are transverse to each other away from a link $\Delta$ in $M$, on which they coincide as oriented $2$-planes. One can further normalize $\xi_+$ and $\xi_-$ on a neighborhood of $\Delta$ so it becomes a {\em normal contact pair}; see \cite[Section 2.3]{CF} for the definition and \cite[Prop.\ 2.1]{CF} which shows that any contact pair can be isotoped into a normal one. In particular, for a normal pair, the link $\Delta$ is transverse to $\xi_\pm$ away from a finite number of points.

If $\alpha_\pm$ is a contact form for $\xi_\pm$, then the smooth vector field $X \in \xi_-$ given by $$i_X d\alpha_-|_{\xi_-} =\alpha_+|_{\xi_-}$$ directs $\xi_+ \cap \xi_-$ and is zero along $\Delta$. For $x\in M$ and $t\in \R$, denote by $\phi_t(x)$ the image of $x$ by the time-$t$ flow of $X$.

\begin{lemma}[\cite{CF}]
If $(\xi_+,\xi_-)$ is a normal contact pair on a closed $3$-manifold $M$, then the plane fields
$$\lambda_\pm^t (x):=(\phi_t)_* \xi_\pm (\phi_{-t} (x))$$
limit to a common plane field $\lambda$ as $t\to +\infty$. Moreover $\lambda$ is continuous and invariant by the flow of the smooth vector field $X$.
\end{lemma}

The plane field $\lambda$ is then locally integrable, but not uniquely since $\lambda$ is only continuous.

\begin{pb}
Try to improve $\lambda$ so it becomes a genuine foliation.  One possible approach is to apply the work of Burago-Ivanov \cite{BI} to prove that $\lambda$ gives a {\em branching foliation}, in which case it could always be deformed into a genuine foliation.
\end{pb}

The integral leaves of $\lambda$ nevertheless reflect the properties of $\xi_+$ and $\xi_-$.

\begin{prop}[\cite{CF}]
If $(\xi_+,\xi_-)$ is a normal tight contact pair on a closed $3$-manifold $M$, then $\lambda$ has no integral $2$-sphere.
\end{prop}

Moreover, we have:

\begin{thm}[\cite{CF}] \label{thm: reeb}
Let $(\xi_+,\xi_-)$ be a normal tight contact pair on a closed $3$-manifold $M$.  If $\lambda$ is uniquely integrable, then the integral foliation $\F$ of $\lambda$ does not contain a Reeb component whose core $c$ is zero in $H_1(M;\Q)$. Moreover, $M$ carries a Reebless foliation.
\end{thm}

\begin{rmk}
Conversely, if the integral foliation $\F$ of $\lambda$ is taut, then the structures $\xi_\pm$ are universally tight, since they can be deformed continuously to $\F$.
\end{rmk}

\begin{pb}
Show that under the hypotheses of Theorem~\ref{thm: reeb}, the foliation $\F$ has no Reeb components (even when $0\not=[c]\in H_1(M;\Q)$).
\end{pb}

A contact pair $(\xi_+,\xi_-)$ is {\em strongly tight} if any two points in $M$ can be joined by an arc that is positively transverse to both $\xi_+$ and $\xi_-$. This condition is equivalent to the property that any two points in $M$ can be joined by an arc positively transverse to $\lambda$.

Suppose $(\xi_+,\xi_-)$ is strongly tight.  If $\lambda$ is uniquely integrable, then its integral foliation $\F$ is taut and $\xi_+$ and $\xi_-$ are universally tight and semi-fillable. 
The key step in the construction of a semi-filling is the existence of a closed $2$-form $\omega$, called a {\em dominating $2$-form}, which satisfies $\omega|_{\F}>0$.  The usual construction of a dominating $2$-form $\omega$ (i.e., summing well-chosen Poincar\'e duals $\omega_\gamma$ supported on small tubular neighborhoods of closed transversals $\gamma$ to $\xi_\pm$) works even when $\lambda$ is not integrable.
Hence $\xi_+$ and $\xi_-$ are semi-fillable even when $\lambda$ is not assumed to be uniquely integrable.

In certain situations, one can prove the integrability of $\lambda$ and derive strong constraints, as in this version of the Reeb stability theorem.

\begin{thm}[\cite{CF}]  \label{thm: sphere}
Let $M$ be a compact connected oriented $3$-manifold with $\partial M=S^2$. If $M$ carries a normal tight contact pair $(\xi_+,\xi_-)$ such that
\be
\item $\Delta$ intersects $\partial M$ transversely in two points $N$ and $S$ and
\item $X$ exits transversely from $M$ along $\partial M\setminus \{ N,S\}$,
\ee
then $M\simeq B^3$.
\end{thm}

The proof of Theorem~\ref{thm: sphere} relies on the fact that, under the assumptions of the theorem, the foliation $\lambda$ is integrable and the leaves are disks.

\begin{pb}\label{pb: irreducible}
Show that if $M$ carries a tight contact pair, then it is irreducible.
\end{pb}

In general, it is not easy to verify the condition in the definition of a contact pair that $\xi_+$ and $\xi_-$ have a common transverse vector field $Z$.

\begin{pb}\label{pb: elimination}
Let $\xi_+$ be a positive contact structure and $\xi_-$ be a negative contact structure such that $\xi_+$ and $\xi_-$ are homotopic as plane fields. Can one isotop $\xi_-$ to $\xi_-'$ so that $(\xi_+,\xi_-')$ is a contact pair? Can this be done with parameters?
\end{pb}

The parametric version of Problem~\ref{pb: elimination} can be solved on the ball, where it is in fact equivalent to Cerf's theorem or Hatcher's theorem:  Let $\zeta_{0,+}$ and $\zeta_{0,-}$ be the standard contact structures on the unit ball $B^3$, with equations $dz\pm r^2 d\theta =0$ in cylindrical coordinates. In this case $X=r\partial_r$ and $\lambda =\{ dz=0\}$. Now take a diffeomorphism $\phi$ of $B^3$ which is the identity on the boundary and consider $\xi_{1,\pm} =\phi_* \xi_{0,\pm}$. By Eliashberg~\cite{El2}, we know that the space of (positive or negative) tight contact structures on $B^3$ that agree with $\xi_{0,\pm}$ along $\bdry B^3$ is contractible, so we can find paths of contact structures $\xi_{t,\pm}$, $t\in [0,1]$, from $\xi_{0,\pm}$ to $\xi_{1,\pm}$. If we could make $(\xi_{t,+},\xi_{t,-})$ into a path of contact pairs by deforming $\xi_{t,-}$, this would immediately give a path $\lambda_t$ of integral foliations by disks from $\lambda_0$ to $\lambda_1$ and then, by standard arguments, an isotopy from $\phi$ to the identity, i.e., an alternative proof of Cerf's theorem --- and with more parameters, a proof of Hatcher's theorem.  On the other hand, we know that such a deformation is possible: it can be constructed by an application of Hatcher's theorem.

\begin{pb}
Prove Hatcher's theorem using the parametric version of Problem \ref{pb: elimination}.
\end{pb}

\begin{pb}
Try to use families of contact pairs to solve Problem \ref{pb: fibration2}.
\end{pb}

\end{document}